\documentclass{article}

\usepackage{amssymb,latexsym,amsmath,stmaryrd,amsfonts,amsthm}
\usepackage{mathabx}
\usepackage{hyperref}
\usepackage{times}

\usepackage{mathrsfs,color,graphicx}
\usepackage[mathscr]{euscript}
\usepackage{indentfirst}

\theoremstyle{plain}
\newtheorem{theorem}{Theorem}[section]
\newtheorem{lemma}[theorem]{Lemma}
\newtheorem{proposition}[theorem]{Proposition}
\newtheorem{corollary}[theorem]{Corollary}

\theoremstyle{definition}

\newtheorem{example}[theorem]{Example}

\theoremstyle{remark}
\newtheorem{remark}[theorem]{Remark}

\newcommand{\R}{\mathbb{R}}

\newcommand{\N}{\mathbb{N}}

\DeclareMathOperator{\dom}{dom}
\DeclareMathOperator{\ran}{ran}
\DeclareMathOperator{\gra}{gra}

\newcommand{\F}{\mathcal{F}}

\newcommand{\tos}{\rightrightarrows} 

\newcommand{\inner}[2]{\langle #1,#2 \rangle}
\DeclareMathOperator{\cone}{cone}
\DeclareMathOperator{\co}{co}

\DeclareMathOperator{\aff}{aff}

\begin{document}

\title{Remarks on $p$-monotone operators
%
}


\author{Orestes Bueno \and John Cotrina\footnote{Universidad del Pac\'ifico, 
              Av. Salaverry 2020, Lima 11, Per\'u            \texttt{o.buenotangoa@up.edu.pe}~and~\texttt{cotrina\_je@up.edu.pe} }}


\maketitle

\begin{abstract}
In this paper, we deal with three aspects of $p$-monotone operators. First we study $p$-monotone operators with a unique maximal extension (called pre-maximal), and with convex graph. We then deal with linear operators, and provide characterizations of $p$-monotonicity and maximal $p$-monotonicity. Finally we show that the Brezis-Browder theorem preserves $p$-monotonicity in reflexive Banach spaces.
\bigskip

\end{abstract}

\section{Introduction and definitions}\label{intro}

Let $U,V$ be non-empty sets. A \emph{multivalued operator} $T:U\tos V$ is an application $T:U\to \mathcal{P}(V)$, that is, for $u\in U$, $T(u)\subset V$. 
The \emph{domain}, \emph{range} and \emph{graph} of $T$ are defined, respectively, as
\begin{gather*}
\dom(T)=\big\{u\in U\::\: T(u)\neq\emptyset\big\},\qquad \ran(T)=\bigcup_{u\in U}T(u)\\ 
\text{and }\gra(T)=\big\{(u,v)\in U\times V\::\: v\in T(u)\big\}.
\end{gather*}
From now on, unless otherwise stated, we will identify multivalued operators with their graphs, so we will write $(u,v)\in T$ instead of $(u,v)\in \gra(T)$. 

Let $V$ be a vector space. If $A\subset V$ then $\co(A)$, $\cone(A)$ and $\aff(A)$ respectively denote the \emph{convex} \emph{conic} and \emph{affine hull} of $A$. In addition, $-A$ denotes the set 
\[
-A=\{x\in V\::\: -x\in A\}.
\]
Recall that if $A$ is convex and $0\in A$ then $\aff(A)=\cone(A)-\cone(A)=\cone(A-A)$.

Let $X$ be a Banach space and $X^*$ be its topological dual. 
The \emph{duality product} $\pi:X\times X^*\to\R$ is defined as $\pi(x,x^*)=\inner{x}{x^*}$. 
Let $A\subset X$, we denote 
\[
A^{\bot}=\{x^*\in X^*\::\:\inner{x}{x^*}=0,\,\forall\,x\in A\}.
\]
Given $h:X\times X^*\to \R\cup\{\infty\}$, denote
\begin{align*}
\{h=\pi\}&=\{(x,x^*)\in X\times X^*\::\:h(x,x^*)=\pi(x,x^*)\} \\
\{h\geq \pi\}&=\{(x,x^*)\in X\times X^*\::\:h(x,x^*)\geq\pi(x,x^*)\}\\
\{h\leq \pi\}&=\{(x,x^*)\in X\times X^*\::\:h(x,x^*)\leq\pi(x,x^*)\}
\end{align*}
Clearly $\{h=\pi\}=\{h\geq \pi\}\cap \{h\leq\pi\}$.

A multivalued operator $T:X\tos X^*$ is called \emph{monotone} if, for every pair $(x,x^*)$, $(y,y^*)\in T$, 
\[
\inner{x-y}{x^*-y^*}\geq 0.
\]
Moreover, $T$ is \emph{maximal monotone} if $T$ is not properly contained in a monotone operator. A \emph{maximal monotone  extension} of $T$ is a maximal monotone operator $M$ which contains $T$. It is an easy consequence of Zorn's Lemma that every monotone operator has a maximal monotone extension.  If $T$ possesses a unique maximal monotone extension, then it is called \emph{pre-maximal monotone}.

Let $T:X\tos X^*$ be a multivalued operator. Then $T$ is \emph{linear} (respectively, \emph{affine linear}) if its graph is a linear (respectively, affine) subspace of $X\times X^*$. It is easy to verify that if $T$ is linear, then it is monotone if, and only if, for every $(x,x^*)\in T$, $\inner{x}{x^*}\geq 0$.

\section{The $p$-monotone polar}
A multivalued operator $T:X\tos X^*$ is called \emph{$p$-monotone}, with $p\in \N$, if  
\[
\{(x_i,x_i^*)\}_{i=0}^p\subset T\quad\longrightarrow\quad\sum_{i=0}^p\langle x_{i+1}-x_i,x_i^*\rangle\leq 0,
\]
where $x_{p+1}=x_0$. An operator is called \emph{cyclically monotone} if it is $p$-monotone, for all $p\in\N$.
If for all $p$-monotone (respectively, cyclically monotone) operators $S$ with $S\supset T$ we have $S=T$, then $T$ is called \emph{maximal $p$-monotone} 
(respectively, maximal cyclically monotone).

Clearly when $p=1$ we recover the classical definition of monotonicity. Also, if an operator is $p$-monotone then it is $q$-monotone, for all $q\leq p$, and in particular, it is monotone.

The \emph{Fitzpatrick function of order $p$}~\cite{BBBRW07} of a multivalued operator $T:X\tos X^*$ is the function 
$\F_{T,p}:X\times X^*\to\R\cup\{\infty\}$ defined for each $(x_0,x_0^*)\in X\times X^*$ as
\[
\F_{T,p}(x_0,x_0^*)=\sup_{\{(x_i,x_i^*)\}_{i=1}^p\subset T} \sum_{i=0}^{p}\inner{x_{i+1}-x_i}{x_i^*}+\inner{x_0}{x_0^*}
\]
considering $x_{p+1}=x_0$.  
In addition, the \emph{Fitzpatrick function of infinite order} is defined as $\F_{T,\infty}=\displaystyle\sup_{p\in\N}\F_{T,p}$.
Note that the Fitzpatrick function of order $1$ is the classical Fitzpatrick function of $T$.

\begin{proposition}[{\cite[Proposition~2.3]{BBBRW07}}]\label{pro:21}\quad
\begin{enumerate}
\item $\F_{T,p}$ is convex and lower semicontinuous.
\item For all $(x,x^*)\in T$, $\F_{T,p}(x,x^*)\geq \inner{x}{x^*}$. Equivalently, $T\subset \{\F_{T,p}\geq \pi\}$.
\end{enumerate}
\end{proposition}

Using the aforementioned Fitzpatrick functions, we define the \emph{$p$-monotone polar} (or simply, \emph{$p$-polar}) of a multivalued operator $T:X\tos X^*$ as
\[
T^{\mu_p}=\{{\F_{T,p}}\leq \pi\}=\{(x,x^*)\in X\times X^*\::\:\F_{T,p}(x,x^*)\leq \inner{x}{x^*}\}.  
\]
In addition, the \emph{cyclically monotone polar} (or \emph{cyclic polar}) of $T$ is the operator $T^{\mu_\infty}=\displaystyle\bigcap_{p\in\N}T^{\mu_p}$. 
Since for each $x$, the map $x^*\mapsto \F_{T,p}(x,x^*)-\inner{x}{x^*}$ is also convex and weak$^*$-lower semicontinuous, 
$T^{\mu_p}(x)$ is convex and weak$^*$-closed. This also holds when $p=\infty$.

The following characterization of the $p$-polars is immediate.
\begin{proposition}\label{pro:caracpol}
Let $T:X\tos X^*$ be a multivalued operator. Then
\[
T^{\mu_p}=\left\{(x_0,x_0^*)\::\: \sum_{i=0}^{p} \inner{x_{i+1}-x_i}{x_i^*}\leq 0,\,\forall\,\{(x_i,x_i^*)\}_{i=1}^p\subset T\text{ with }x_{p+1}=x_0\right\}.
\]
\end{proposition}
The name ``polar'' is motivated from the fact that, when $p=1$, the $1$-cyclic monotone polar $T^{\mu_1}$ is simply the well known monotone polar $T^{\mu}$~\cite{
BSML}.  However, these ``polars'' are not \emph{polarities} in the sense of~\cite{Birk79}, and do not have the same properties the monotone polar (and also similar notions, such as the quasimonotone and pseudomonotone polars~ \cite{BueCot2017,BueCot16-2}) has.  The following proposition subsumes some of those properties.

\begin{proposition}\label{p-monotone-relation}
Let $T,T_i:X\tos X^*$ be multivalued operators and $p\in\N$. The following hold.
\begin{enumerate}
 \item $T^{\mu_{p+1}}\subset T^{\mu_p}$;
 \item If $T_1\subset T_2$ then $T_2^{\mu_p}\subset T_1^{\mu_p}$;
 \item $\displaystyle\left(\bigcup_{i\in I}T_i\right)^{\mu_p}\subset \bigcap_{i\in I}T_i^{\mu_p}$;
 \item The graph of $T^{\mu_p}$ is (strongly-)closed.
 \item Let $C=\co(\dom(T))$. If $T$ is $p$-monotone then $T+N_C\subset T^{\mu_p}$.
 \item For any $(x,x^*)\in X\times X^*$,  $(T+(x,x^*))^{\mu_p}=T^{\mu_p}+(x,x^*)$.
\end{enumerate}
\end{proposition}
\begin{proof}\quad 
\begin{enumerate}
\item Take $(x_0,x_0^*)\in T^{\mu_{p+1}}$ and let $\{(x_i,x_i^*)\}_{i=1}^p\subset T$ and $x_{p+1}=x_0$. 
Thus
\begin{align*}
\sum_{i=0}^p\inner{x_{i+1}-x_i}{x_i^*}&=\sum_{i=0}^{p-1}\inner{x_{i+1}-x_i}{x_i^*}+\inner{x_0-x_p}{x_p^*}\\
&=\sum_{i=0}^{p-1}\inner{x_{i+1}-x_i}{x_i^*}+\inner{x_p-x_p}{x_p^*}+\inner{x_0-x_p}{x_p^*}\leq 0.
\end{align*}
This implies $(x_0,x_0^*)\in T^{\mu_{p}}$
\item If $T_1\subset T_2$ then $\F_{T_1,p}\leq \F_{T_2,p}$. Hence, $\{\F_{T_2,p}\leq \pi\}\subset \{\F_{T_1,p}\leq \pi\}$.
\item It follows directly from {\it 3}.
\item It is an easy consequence from Proposition~\ref{pro:caracpol}.
\item Let $x\in\dom(T)$ and take $x^*\in T(x)$ and $y^*\in N_C(x)$. Then, for each $\{(x_i,x_i^*)\}_{i=1}^p\subset T$, we have 
\begin{align}\label{p-mono-T}
\langle x_1-x,x^* \rangle+\sum_{i=1}^{p}\langle x_{i+1}-x_i,x_i^*\rangle\leq0,\\\label{normal-cone}
\langle x_1-x,y^*\rangle\leq0,
\end{align}
where $x_{p+1}=x$. We now add \eqref{p-mono-T} and \eqref{normal-cone} to obtain
\[
\langle x_1-x,x^*+y^* \rangle+\sum_{i=1}^{p}\langle x_{i+1}-x_i,x_i^*\rangle\leq0. 
\]
Therefore, $x^*+y^*\in T^{\mu_p}(x)$.
\item It is a straightforward calculation using the definition of $p$-polar.
\qedhere
\end{enumerate}
\end{proof}

We must remark that, when $p=1$, the converse inclusion in item {\it 4} of the previous proposition also holds~\cite{BSML}. 

\begin{example}\label{p=>1}
Let $(u,u^*)\in X\times X^*$. Then all the $p$-monotone polars of $\{(u,u^*)\}$ coincide with its monotone polar, that is
\[
\{(u,u^*)\}^{\mu}=\{(u,u^*)\}^{\mu_p}=\{(u,u^*)\}^{\mu_{\infty}}.
\]
More generally, if $T$ has exactly $p$ elements then $T^{\mu_p}=T^{\mu_{p+1}}=\cdots=T^{\mu_\infty}$.
\end{example}

%

In~\cite{BBBRW07}, the authors showed that a multivalued operator $T$ is $p$-monotone if, 
and only if, $T\subset T^{\mu_p}$. Also, a multivalued operator $T$ is maximal $p$-monotone if, and only if, $T=T^{\mu_p}$. 
So, all maximal $p$-monotone operator has convex and weak$^*$-closed values.

We now deal with maximal $p$-monotone extensions of a $p$-monotone operator.  Given $T:X\tos X^*$ $p$-monotone, 
let $M_p(T)$ be the family of its maximal $p$-monotone extensions.

\begin{proposition}\label{polar+maximal-p}
Let $T:X\tos X^*$ be a multivalued operator. If $T$ is $p$-monotone then the following hold:
\begin{enumerate}
 \item $T^{\mu_p}=\displaystyle\bigcup_{S\in M_p(T)}S$,
 \item $T^{\mu_p\mu_p}\subset\displaystyle\bigcap_{S\in M_p(T)}S$,
 \item $T^{\mu_p\mu_p}$ is $p$-monotone.
\end{enumerate}
\end{proposition}
\begin{proof}
\begin{enumerate}
 \item The result follows from the fact that any $p$-monotone operator 
has a maximal $p$-monotone extension, due to Zorn's Lemma.
\item  Let $S$ a maximal $p$-monotone  extension of $T$. By item~{\it 1}, $S\subset T^{\mu_p}$ and, since $S$ is maximal $p$-monotone, 
$T^{\mu_p\mu_p}\subset S^{\mu_p}=S$. Thus the inclusion follows.
\item Follows from item {\it 2} and by noting that $\displaystyle\bigcap_{S\in M_p(T)}S$ is $p$-monotone.\qedhere
\end{enumerate}
\end{proof}

In \cite{BSML}, the authors showed that if $T$ is monotone then $T\subset T^{\mu_1\mu_1}$, but this is not true in general for $p>1$, consider for instance 
$T=\{(0,0)\}$, which is $2$-monotone and, by Example \ref{p=>1},  
$T^{\mu_2}=T^{\mu_1}=\{(x,y)\in\R^2\::\: xy\geq 0\}$. However, $T^{\mu_2\mu_2}=\emptyset$. 

\begin{proposition}
Let $T:X\tos X^*$ be a multivalued operator. If $T$ is $p$-monotone then 
$T^{\mu_p\mu_p}=\{\F_{T^{\mu_p},p}=\pi\}$. 
\end{proposition}
\begin{proof}
First note that $\{\F_{T^{\mu_p},p}=\pi\}\subset \{\F_{T^{\mu_p},p}\leq\pi\}=T^{\mu_p\mu_p}$. 
On the other hand, $T\subset T^{\mu_p}$ so $T^{\mu_p\mu_p}\subset T^{\mu_p}$.  By proposition~\ref{pro:21}, item {\it 2},
\[
T^{\mu_p\mu_p}\subset T^{\mu_p}\subset \{\F_{T^{\mu_p},p}\geq \pi\},
\]
and the proposition follows.
\end{proof}

The following result is analogous to Proposition 24 in \cite{BSML}.
\begin{proposition}
Let $T:X\tos X^*$ be a multivalued operator such that $T\subset T^{\mu_p\mu_p}$. 
If there exists a $p$-monotone operator $S:X\tos X^*$  such that $\F_{T,p}=\F_{S,p}$
then $T$ is $p$-monotone.
\end{proposition}
\begin{proof}
Clearly, $T^{\mu_p}=\{\F_{T,p}\leq \pi\}=\{\F_{S,p}\leq\pi\}=S^{\mu_p}$. Thus,
\[
T\subset T^{\mu_p\mu_p}=S^{\mu_p\mu_p}\subset S^{\mu_p}=T^{\mu_p}
\]
and, therefore, $T$ is $p$-monotone.
\end{proof}

In~\cite[Example~2.16]{BBBRW07} (see also~\cite{Bauschke08}), the authors showed an example of a 
maximal $2$-monotone operator which is not maximal monotone. This means that, in general, 
maximal $p$-monotonicity does not imply maximal $q$-monotonicity when $q<p$. The next proposition studies the case when $q\geq p$.

\begin{proposition}\label{pro:pmonqmon}
Let $T:X\tos X^*$ be a maximal $p$-monotone operator.  The following hold.
\begin{enumerate}
\item $T^{\mu_{q}}$ is $q$-monotone for all $q\geq p$.
\item If $q\geq p$ and $T$ is $q$-monotone then $T$ is maximal $q$-monotone. 
\end{enumerate}
\end{proposition}
\begin{proof}\quad
\begin{enumerate}
\item By item~{\it 2} of Proposition~\ref{p-monotone-relation}, for each $q\geq p$, we have $T^{\mu_{q}}\subset T^{\mu_p}=T$. 
Taking $q$-monotone polars and again by Proposition~\ref{p-monotone-relation}, we obtain $T^{\mu_q}\subset T^{\mu_q\mu_q}$, 
that is $T^{\mu_q}$ is $q$-monotone. 
\item If $T$ is $q$-monotone, with $q\geq p$, then 
\[
T\subset T^{\mu_q}\subset T^{\mu_p}=T,
\]
and we are done.\qedhere
\end{enumerate}
\end{proof}


\subsection{Pre-maximal $p$-monotone operators}
If $T$ is a $p$-monotone operator such that $T^{\mu_p}$ is also $p$-monotone 
(therefore, maximal $p$-monotone), then $T$ will be called \emph{pre-maximal $p$-monotone}. A pre-maximal $p$-monotone operator 
is not necesarily maximal, but possesses a unique maximal $p$-monotone extension.

From Proposition~\ref{polar+maximal-p}, we easily obtain the following corollary.
\begin{corollary}\label{cor:premax}
	If $T^{\mu_p\mu_p}$ is maximal $p$-monotone then $T^{\mu_p\mu_p}=T^{\mu_p}$, $T$ is pre-maximal monotone and $T^{\mu_p}$ is its unique maximal monotone extension.
\end{corollary}

\begin{proposition}\label{pre-maximal}
	Let $T,S:X\tos X^*$ be multivalued operators. If $T$ is pre-maximal $p$-monotone and $T\subset S\subset T^{\mu_p}$ then $S$ is pre-maximal $p$-monotone and $S^{\mu_p}=T^{\mu_p}$.
\end{proposition}
\begin{proof}
Since $S\subset T^{\mu_p}$ and $T^{\mu_p}$ is $p$-monotone, we obtain that $S$ is $p$-monotone. By Proposition \ref{p-monotone-relation}, item {\it 3}, we have $T\subset S\subset S^{\mu_p}\subset T^{\mu_p}$. Thus, $S^{\mu_p}$ is $p$-monotone.
	Finally, maximal $p$-monotonicity of  $S^{\mu_p}$ implies $S^{\mu_p}=T^{\mu_p}$.
\end{proof}
The following proposition extends the Proposition 36 in \cite{BSML}.
\begin{proposition}
	Let $T:X\tos X^*$ be a $p$-monotone operator. Then the following conditions are equivalent:
	\begin{enumerate}
		\item$T$ is pre-maximal $p$-monotone.
		\item $T^{\mu_p}=T^{\mu_p\mu_p}$.
		\item $\F_{T^{\mu_p},p}=\F_{T^{\mu_p\mu_p},p}$.
	\end{enumerate}
\end{proposition}
\begin{proof}
	If $T$ is pre-maximal $p$-monotone, then $T^{\mu_p}$ is maximal $p$-monotone and this implies $T^{\mu_p}=T^{\mu_p\mu_p}$ and, consequently,  $\F_{T^{\mu_p},p}=\F_{T^{\mu_p\mu_p},p}$.  This proves implications {\it 1} $\to$ {\it 2} and {\it 2} $\to$ {\it 3}.
	
	Now we show {\it 3} $\to$ {\it 1}. It follows from the definition that  $T^{\mu_p\mu_p}=T^{\mu_p\mu_p\mu_p}$ which in turn implies that 
	$T^{\mu_p\mu_p}$ is maximal $p$-monotone. By Corollary~\ref{cor:premax}, $T^{\mu_p}=T^{\mu_p\mu_p}$ and, therefore, $T$ is pre-maximal $p$-monotone.
\end{proof}

\subsection{On $p$-monotone operators with convex graph}
Throughout this section, we will drop the identification of a multivalued operator with its graph. 
Let $T:X\tos X^*$ be a multivalued operator. Denote by $T_{\cone}$ to the multivalued operator from $X$ to $X^*$ defined via its graph as
\[
\gra(T_{\cone})=\cone(\gra(T)).
\]
Note that $\gra(T)\subset\gra(T_{\cone})$ and, clearly, if either $T_{\cone}$ is $p$-monotone then $T$ is too. The converse is not true in general,
consider for instance $T=f:\R\to\R$, $f(x)=x^3$ which is monotone, but $T_{\cone}$ is not monotone.

We now give sufficient conditions in order to establish the $p$-monotonicity of $T_{\cone}$.
\begin{proposition}\label{cone-p-cyclic}
Let $T:X\tos X^*$ be a $p$-monotone operator with convex graph such that $(0,0)\in\gra(T)$. Then $T_{\cone}$ is $p$-monotone and $-\gra(T_{\cone})\subset \gra(T^{\mu_p})$.
\end{proposition}
\begin{proof}
Let $\{(x_i,x_i^*)\}_{i=0}^p\subset \gra(T_{\cone})$, then there exist $\{(y_i,y_i^*)\}_{i=0}^p\subset \gra(T)$ and $\{\alpha_i\}_{i=0}^p\subset\R_+$ such that $(x_i,x_i^*)=\alpha_i(y_i,y_i^*)$, for all $i\in\{0,\ldots,p\}$.
Consider $m=\displaystyle\sum_{i=0}^p\alpha_i$, since $\gra(T)$ is convex and $(0,0)\in \gra(T)$,
\[
(x_i/m,x_i^*/m)=(\alpha_i/m)(y_i,y_i^*)+(1-\alpha_i/m)(0,0)\in \gra(T).
\]
Hence,
\[
\sum_{i=0}^p\langle x_{i+1}-x_i,x_i^*\rangle=
m^2  \sum_{i=0}^p\left\langle \frac{x_{i+1}}{m}-\frac{x_i}{m},\frac{x_i^*}{m}\right\rangle\leq0.
\] 
and $T_{\cone}$ is $p$-monotone.

We now show that $-\gra(T_{\cone})\subset \gra(T^{\mu_p})$. Let $(-x_0,-x^*_0)\in \gra(T_{\cone})$ and $\{(x_i,x_i^*)\}_{i=0}^p\subset \gra(T)$.
As the graph of $T_{\cone}$ is a convex cone, define
$(y_i,y_i^*)=2(-x_0,-x_0^*)+(x_i,x_i^*)\in \gra(T_{\cone})$, for all $i\in\{1,2,\dots,p\}$. Thus,
\[
\sum_{i=0}^p\inner{x_{i+1}-x_i}{x_i^*}=\sum_{i=0}^p\inner{y_{i+1}-y_i}{y_i^*}\leq 0
\]
where $x_{p+1}=x_0$ and $y_{p+1}=y_0$. Therefore, $(x_0,x_0^*)\in \gra(T^{\mu_p})$ and the proposition follows.
\end{proof}
\begin{corollary}
Let $T:X\tos X^*$ be a $p$-monotone operator with convex graph such that $(0,0)\in \gra(T)$. Then $\aff(\gra(T))\subset \gra(T^{\mu_p})$.
\end{corollary}
\begin{proof}
Recall that, since $(0,0)\in \gra(T)$, $\aff(\gra(T))=\gra(T_{\cone})-\gra(T_{\cone})$. So take $(x,x^*),(y,y^*)\in \gra(T_{\cone})$, we aim to prove that $(x,x^*)-(y,y^*)\in \gra(T^{\mu_p})$. Let $\{(x_i,x_i^*)\}_{i=0}^p\subset\gra(T)$ and define $(x_0,x_0^*)=(x,x^*)-(y,y^*)$,
\[
(y_i,y_i^*)=(x_i,x_i^*)+(y,y^*),\quad i\in\{1,\ldots,p\},
\]
and $(y_0,y_0^*)=(x,x^*)$. Since $\gra(T_{\cone})$ is a convex cone, $(y_i,y_i^*)\in \gra(T_{\cone})$, for all $i$. Therefore, as $T_{\cone}$ is $p$-monotone, by Proposition~\ref{cone-p-cyclic},
\[
\sum_{i=0}^p\inner{y_{i+1}-y_i}{y_i^*}\leq 0.
\]
The corollary now follows by noting that
\[
\sum_{i=0}^p\inner{x_{i+1}-y_i}{x_i^*}=\sum_{i=0}^p\inner{y_{i+1}-y_i}{y_i^*}\leq 0,
\]
that is $(x,x^*)-(y,y^*)\in \gra(T^{\mu_p})$.
\end{proof}
Observe that both the $p$-monotone polar and the affine hull of the graph of an operator are preserved by translations. Therefore, we readily obtain the following corollary.
\begin{corollary}\label{cor:convexaffine}
Let $T:X\tos X^*$ be a $p$-monotone operator with convex graph. Then $\aff(\gra(T))\subset \gra(T^{\mu_p})$.
\end{corollary}

\section{On linear $p$-monotone operators}
In this section we deal with linear operators.  We begin with a couple of lemmas.
\begin{lemma}\label{span-inclusion-polar}
	Let $T:X\tos X^*$ be a linear operator and let $(x,x^*)\in T^{\mu_p}$. Then, for all $t\in\R$, $t(x,x^*)\in T^{\mu_p}$.
\end{lemma}
\begin{proof}
	Let $\{(x_i,x_i^*)\}_{i=1}^p\subset T$ and $t\in\R$. Assume $t\neq 0$, since $T$ is linear then $(y_i,y_i^*)=(x_i/t,x_i^*/t)\in T$. Thus,
	\begin{multline*}
	\inner{x_1-tx}{tx^*}+\sum_{i=1}^p\inner{x_{i+1}-x_i}{x_i^*}+\inner{tx-x_p}{x_p^*} 
	=\\
	t^2\left(\inner{y_1-x}{x^*}+\sum_{i=1}^p\inner{y_{i+1}-y_i}{y_i^*}+\inner{x-y_p}{y_p^*}\right)\leq 0.
	\end{multline*}
	This implies $(tx,tx^*)\in T^{\mu_p}$, when $t\neq 0$. Using the fact that $T^{\mu_p}$ is closed, we conclude that $(0,0)\in T^{\mu_p}$. The lemma follows.
\end{proof}
\begin{lemma}\label{lem:domt}
	Let $T:X\tos X^*$ be a linear monotone operator. Then $T^{\mu}(0)\subset\dom(T)^{\bot}$.
\end{lemma}
\begin{proof}
	Take $x^*\in T^{\mu}(0)$ and $(y,y^*)\in T$. Since $T$ is linear, for every $t\in\R$, 
	\[
	\inner{0-ty}{x^*-ty^*}\geq 0,
	\]
	that is, $t^2\inner{y}{y^*}-t\inner{y}{x^*}\geq 0$, for all $t\in\R$. This implies $\inner{y}{x^*}=0$ and, as $y$ was taken arbitrarily, $x^*\in\dom(T)^\bot$.  
\end{proof} 

\begin{theorem}\label{charac-linear}
Let $T:X\tos X^*$ be a linear operator. The following are equivalent:
\begin{enumerate}
\item $T$ is $p$-monotone,
\item $T^{\mu_p}(0)=\dom(T)^{\bot}$,
\item  $T^{\mu_p}\neq\emptyset$,
\item $(0,0)\in T^{\mu_p}$.
\end{enumerate}
\end{theorem}
\begin{proof}
We first show {\it 1} $\to$ {\it 2}. By Proposition~\ref{p-monotone-relation}, item {\it 2}, and Lemma~\ref{lem:domt}, $T^{\mu_p}(0)\subset T^{\mu_1}(0)\subset\dom(T)^{\bot}$.  
Now let $u_0^*\in \dom(T)^{\bot}$ and $u_0=u_{p+1}=0$.  Taking $(u_1,u_1^*),\ldots,(u_p,u_p^*)\in T$, 
\begin{align*}
\sum_{i=0}^p\inner{u_{i+1}-u_i}{u_i^*}&=\inner{u_{1}-u_0}{u_0^*}+\sum_{i=1}^p\inner{u_{i+1}-u_i}{u_i^*}\\
&=\inner{u_{1}-0}{0}+\sum_{i=1}^p\inner{u_{i+1}-u_i}{u_i^*}\leq 0,
\end{align*}
where the second equality holds since $u_1\in\dom(T)$ and the previous inequality comes from the fact that 
$\{(0,0)\}\cup\{(u_i,u_i^*)\}_{i=1}^p\subset T$ and $T$ is $p$-monotone.
	
The implications {\it 2} $\to$ {\it 4} and {\it 4} $\to$ {\it 3} are trivial. 
The implication {\it 3} $\to$ {\it 4} is a direct consequence of Lemma \ref{span-inclusion-polar}. We conclude by showing {\it 4} $\to$ {\it 1}. Let $\{(x_i,x_i^*)\}_{i=0}^p\subset T$ with $x_{p+1}=x_0$ and define $(y_i,y_i^*)=(x_i-x_0,x_i^*-x_0^*)$, for $i=0,\ldots,p$. Clearly $\{(y_i,y_i^*\}_{i=0}^p\subset T$, since $T$ is linear. Now, as $(y_0,y_0^*)=(0,0)\in T^{\mu_p}$, using Proposition~\ref{pro:caracpol},
\[
0\geq \sum_{i=0}^p\inner{y_{i+1}-y_i}{y_i^*}=\sum_{i=0}^p\inner{x_{i+1}-x_i}{x_i^*-x_0^*}=\sum_{i=0}^p\inner{x_{i+1}-x_i}{x_i^*},
\]
where the equality follows from $\displaystyle\sum_{i=0}^p\langle x_{i+1}-x_i,x_0^*\rangle=0$. This implies that $T$ is $p$-monotone.
\end{proof}


As a consequence of Theorem~\ref{charac-linear}, we obtain the following theorem.
\begin{theorem}\label{max-carac-cero}
Let $T:X\tos X^*$ be a linear operator. The following hold.
\begin{enumerate}
\item If $T$ is maximal $p$-monotone then $T(0)=\dom(T)^{\bot}$.
\item If $T$ is $p$-monotone with closed domain and $T(0)=\dom(T)^{\bot}$ then $T$ is maximal monotone and, in particular, maximal $p$-monotone.
\end{enumerate}
\end{theorem}
\begin{proof}
\begin{enumerate}
\item By Theorem~\ref{charac-linear}, $T(0)=T^{\mu_p}(0)=\dom(T)^{\bot}$.
\item This is an immediate consequence of~\cite[Proposition 5.2]{BorBau2011} and Proposition~\ref{pro:pmonqmon}.
\end{enumerate}
\end{proof}

\begin{example}\label{ex:rlinmax}
Let $T$ be a linear operator with closed domain, and let $N_D=N_{\dom(T)}$ be the normal cone operator associated to the domain of $T$. Consider $R=T+N_D$, then $R$ is also linear and $p$-monotone. It is easy to verify that $\dom(R)=\dom(T)$, therefore  $R(0)=T(0)+N_D(0)=\dom(T)^{\bot}=\dom(R)^{\bot}$. By Theorem~\ref{max-carac-cero}, $R$ is maximal monotone and maximal $p$-monotone. Note that the graph of $R$ can be characterized as a sum of linear subspaces of $X\times X^*$, namely,
\[
\gra(R)=\gra(T)+(\{0\}\times\dom(T)^{\bot}).
\]
\end{example}

We now aim to characterize the pre-maximal $p$-monotonicity of linear operators. First, one lemma.

\begin{lemma}\label{mono-linear-included}
Let $T:X\tos X^*$ be a linear monotone operator. If $S:X\tos X^*$ is linear and
\[
 T\subset S\subset T^\mu,
\]
then $S$ is also monotone.
\end{lemma}
\begin{proof}
Take $(x,x^*)\in S$, then $(x,x^*)\in T^{\mu}$ and, since $(0,0)\in T$, 
\[
\inner{x}{x^*}=\inner{x-0}{x^*-0}\geq 0.
\]
As $(x,x^*)\in S$ was arbitrary, we conclude that $S$ is monotone. 
\end{proof}

From the previous lemma, note that the linearity of $T^{\mu}$ implies pre-maximal monotonicity of $T$. We now generalize this fact to linear $p$-monotone operators.

\begin{proposition}
Let $T:X\tos X^*$ be a linear $p$-monotone operator with closed domain. Then $T$ is pre-maximal $p$-monotone if, and only if, $T^{\mu_p}$ is linear. 
\end{proposition}
\begin{proof}
As $T$ has closed domain, by Example~\ref{ex:rlinmax}, $R=T+N_D$ is a linear maximal $p$-monotone extension of $T$. Even more, $R$ is also maximal monotone.
	
Assume that $T$ is pre-maximal $p$-monotone. Then $R$ must be the unique maximal $p$-monotone extension of $T$. Therefore $T^{\mu_p}=R$ is linear.

Conversely, assume that $T^{\mu_p}$ is linear. Since $T\subset T^{\mu_p}\subset T^{\mu}$, by Lemma~\ref{mono-linear-included}, $T^{\mu_p}$ is monotone. On the other hand, $R\subset T^{\mu_p}$ and, since $R$ is maximal monotone, we have $R=T^{\mu_p}$ and $T^{\mu_p}$ is $p$-monotone. Therefore $T$ is pre-maximal $p$-monotone. 
\end{proof}
\begin{corollary}
Let $T:X\tos X^*$ be a linear $p$-monotone operator with closed domain and let $q<p$. If $T$ is pre-maximal $q$-monotone then it is pre-maximal $p$-monotone.
\end{corollary}
\begin{proof}
Consider $R=T+N_D$ as in Example~\ref{ex:rlinmax}, which is maximal monotone. As $T$ is $p$-monotone and pre-maximal $q$-monotone, $T^{\mu_q}$ is $q$-monotone (hence monotone) and
\[
R\subset T^{\mu_p}\subset T^{\mu_q}.
\]
Therefore $R=T^{\mu_q}=T^{\mu_p}$ is linear and $T$ is pre-maximal $p$-monotone.
\end{proof}

We conclude this section with the following Theorem, which relates to Theorem~4.2 in \cite{BauWanYao2009} (see also Lemma~1.2 in \cite{BM3}.)
\begin{theorem}
Let $T:X\tos X^*$ be a maximal $p$-monotone operator with convex graph. Then $T$ is affine linear. Moreover, if $\dom(T)$ is closed then $T$ is maximal monotone.
\end{theorem}
\begin{proof}
By Corollary~\ref{cor:convexaffine}, $T\subset \aff(T)\subset T^{\mu_p}=T$. Therefore $T=\aff(T)$, that is, $T$ is affine linear. If $\dom(T)$ is closed, translating $T$ to the origin if necessary, Theorem~\ref{max-carac-cero} implies that $T$ is maximal monotone.
\end{proof}

\section{The Br\'ezis-Browder Theorem}
From now on, we will assume that $X$ is a \emph{reflexive} Banach space and we will identify the \emph{bidual} space $X^{**}$ with $X$.

Given $p\in\N$, consider $X_p$ to be the cartesian product of $X$ taken $p+1$ times, namely 
\[
X_p=\prod_{i=0}^p X,
\]
and endow $X_p$ with the norm $\|(x_0,\ldots,x_p)\|_p=(\|x_0\|^2+\cdots+\|x_p\|^2)^{1/2}$.
Moreover, identify $X_p^*=(X_p)^*=\displaystyle\prod_{i=0}^p X^*$.  The duality coupling between $X_p$ and $X_p^*$ is given by
\[
\inner{\bar x}{\bar x^*}_p=\sum_{i=0}^p\inner{x_i}{x_i^*},
\]
where $\bar x=(x_0,\ldots,x_p)\in X_p$ and $\bar x^*=(x_0^*,\ldots,x_p^*)\in X_p^*$. Note that $X_p$ is reflexive, since $X$ is.

Let $T:X\tos X^*$ be a multivalued operator. Denote by $T_{p\pm}:X_p\tos X_p^*$ to the multivalued operators defined by
\[
T_{p\pm}(x_0,\dots,x_p)=\prod_{i=0}^p T(x_i)-T(x_{i\pm 1}),
\]
considering $x_{p+1}=x_{-1}=x_p$. Equivalently, $(\bar x,\bar x^*)\in T_{p\pm}$ if, and only if, for all $i\in\{0,\ldots,p\}$, there exist $u_i^*,v_i^*\in T(x_i)$ such that
\[
x_i^*=u_i^*-v_{i\pm 1}^*.
\]
It is clear from the definition of $T_{p\pm}$ that $\dom(T_{p\pm})=\displaystyle\prod_{i=0}^p\dom(T)$.
\begin{proposition}\label{pro:tplinear}
Let $T:X\tos X^*$ be a linear multivalued operator. Then $T_{p\pm}$ is linear. Moreover, $(\bar x,\bar x^*)\in T_{p\pm}$ if, and only if, for all $i\in\{0,\ldots,p\}$, there exist $z_i^*\in T(x_i)$, $\alpha_i^*\in T(0)$, such that
\[
x_i^*=z_i^*-z_{i\pm 1}^*+\alpha_i^*.
\]
\end{proposition}

The following Lemma can be found in~\cite[Proposition 2.1]{JC}. Although it is presented in a finite dimensional setting, the proof is completely analogous.
\begin{lemma}\label{lem:Tinverse}
Let $T:X\tos X^*$ be a multivalued operator. Then $T$ is $p$-monotone if, and only if, $T^{-1}$ is $p$-monotone.
\end{lemma}
\begin{remark}
The previous lemma is trivial for the monotone case, since, given $(x_0,x_0^*)$, $(x_1,x_1^*)\in X\times X^*$,
\[
\sum_{i=0}^1\inner{x_{i+1}-x_i}{x^*_i}=\sum_{i=0}^1\inner{x_i}{x^*_{i+1}-x^*_i}=-\inner{x_0-x_1}{x_0^*-x_1^*}.
\]
The first equality in the previous equation, however, is not true in general when $p>1$. Consider for instance $p=2$, $(x_0,x_0^*)=(0,0)$, $(x_1,x_1^*)=(1,1)$ and $(x_2,x_2^*)=(-1,1)$ in $\R^2$. Then
\[
\sum_{i=0}^p\inner{x_{i+1}-x_{i}}{x^*_{i}}=-1\leq 0,\quad\text{ but }\quad \sum_{i=0}^p\inner{x_i}{x^*_{i+1}-x^*_i}=1>0.
\]
\end{remark}

The starting point of our current analysis is the following proposition.
\begin{proposition}\label{uno}
Let $T:X\tos X^*$ be a multivalued operator. If either $T_{p+}$ or $T_{p-}$ is monotone then $T$ is $p$-monotone. 
\end{proposition}
\begin{proof}
Let $\{(x_i,x_i^*)\}_{i=0}^{p}\subset T$ and define
\[
\bar x=(x_0,\ldots, x_p)\in X_p,\qquad \bar x^*=(x_i-x_{i\pm 1})_{i=0}^p
\in X_p^*.
\]
Clearly $(\bar x,\bar x^*)\in T_{p\pm}$. In addition, fix $(a,a^*)\in T$ and define $\bar y=(a,\ldots,a)\in X_p$ and $\bar y^*=(0,\ldots,0)\in X_p$. Then $(\bar y,\bar y^*)\in T_{p\pm}$ also, and
\begin{equation}\label{eq:prouno}
\inner{\bar x-\bar y}{\bar x^*-\bar y^*}_p=\sum_{i=0}^p\inner{x_i-a}{x_i^*-x_{i\pm1}^*}=\sum_{i=0}^p\inner{x_i}{x_i^*-x_{i\pm 1}^*},
\end{equation}
since $\displaystyle\sum_{i=0}^p\inner{a}{x_i^*-x_{i\pm1}^*}=0$. 

If $T_{p-}$ is monotone, using~\eqref{eq:prouno}, we obtain
\[
\sum_{i=0}^p\inner{x_{i+1}-x_i}{x_i^*}=-\sum_{i=0}^p\inner{x_i}{x_i^*-x_{i-1}^*}=-\inner{\bar x-\bar y}{\bar x^*-\bar y^*}_p\leq 0.
\]
Therefore $T$ is $p$-monotone. On the other hand if $T_{p+}$ is monotone, using again~\eqref{eq:prouno} we obtain 
\[
\sum_{i=0}^p\inner{x_i}{x^*_{i+1}-x^*_i}=-\sum_{i=0}^p\inner{x_i}{x_i^*-x_{i+1}^*}=-\inner{\bar x-\bar y}{\bar x^*-\bar y^*}_p\leq 0.
\]
Hence $T^{-1}$  is $p$-monotone. By Lemma~\ref{lem:Tinverse}, $T$ is $p$-monotone.
\end{proof}
The converse implication of the previous proposition can be obtained under linearity of the operator.
\begin{proposition}
Let $T:X\tos X^*$ be a linear multivalued operator. If $T$ is $p$-monotone then $T_{p+}$ and $T_{p-}$ are monotone.
\end{proposition}
\begin{proof}
By Proposition~\ref{pro:tplinear}, both $T_{p+}$ and $T_{p-}$ are linear. Take $(\bar x,\bar x^*)\in T_{p\pm}$. Again by Proposition~\ref{pro:tplinear}, there exist $z_i^*\in T(x_i)$, $\alpha_i^*\in T(0)$ such that
\[
x_i^*=z_i^*-z_{i\pm 1}^*+\alpha_i^*.
\]
Hence
\begin{align*}
\inner{\bar x}{\bar x^*}_p&=\sum_{i=0}^p\inner{x_i}{x_i^*}\\
&=\sum_{i=0}^p\inner{x_i}{z_i^*-z_{i\pm 1}^*+\alpha_i^*}\\
&=\sum_{i=0}^p\inner{x_i}{z_i^*-z_{i\pm 1}^*},
\end{align*}
since $\alpha_i^*\in T(0)\subset \dom(T)^{\bot}$ and $x_i\in\dom(T)$. Monotonicity of $T_{p-}$ (respectively, of $T_{p+}$) follows directly from the $p$-monotonicity of $T$ (respectively $T^{-1}$).
\end{proof}

\begin{proposition}\label{pro:maxtp}
Let $T:X\tos X^*$ be a linear operator with closed domain. The following are equivalent
\begin{enumerate}
 \item $T$ is maximal $p$-monotone.
 \item $T_{p+}$ is maximal monotone.
 \item $T_{p-}$ is maximal monotone.
\end{enumerate}
\end{proposition}
\begin{proof}
Note that, in general,
\[
T_{p\pm}(0)=\prod_{i=0}^p T(0),\qquad \dom(T_{p\pm})^{\bot}=\prod_{i=0}^p \dom(T)^{\bot}.
\]
The equivalence between {\it 1}, {\it 2} and {\it 3} follows from Theorem~\ref{max-carac-cero} and the closedness of the domains.
\end{proof}

The \emph{adjoint operator} associated to $T:X\tos X^*$ is the operator $T^*:X\tos X^*$ defined as
\begin{equation}\label{eq:adjoint}
(x,x^*)\in T^*\quad\iff\quad \inner{x}{y^*}=\inner{y}{x^*},\quad\forall(y,y^*)\in T.
\end{equation}
It is straightforward to verify that $T^*$ is linear and strongly closed on $X\times X^*$.

We now recall a classical result due to Brezis and Browder~\cite{}. 
\begin{theorem}\label{teo:bb}
Let $X$ be a reflexive Banach space and $T:X\tos X^*$ be a linear operator with closed graph. Then the following are equivalent:
\begin{enumerate}
 \item $T$ is maximal monotone.
 \item $T^*$ is maximal monotone.
 \item $T^*$ is monotone.
\end{enumerate}
\end{theorem}
Our goal is to prove a $p$-monotone version of this theorem. This was addressed previously in~\cite{JC}, for Euclidean spaces.

\begin{proposition}
Let $T:X\tos X^*$ be a multivalued operator. Then $(T^*)_{p+}\subset (T_{p-})^*$ and $(T^*)_{p-}\subset (T_{p+})^*$.
\end{proposition}
\begin{proof}
Let $(\bar x,\bar x^*)\in (T^*)_{p+}$. Then there exist $z_i^*,w_i^*\in T^*(x_i)$ such that $x_i^*=z_i^*-w_{i+1}^*$. Now take any $(\bar y,\bar y^*)\in T_{p-}$, and let $u_i,v_i\in T(y_i)$ such that $y_i^*=u_i^*-v_{i-1}^*$. Therefore
\begin{align*}
\inner{\bar x}{\bar y^*}_p&=\sum_{i=0}^p\inner{x_i}{y_i^*}=\sum_{i=0}^p\inner{x_i}{u_i^*-v^*_{i-1}}=\sum_{i=0}^p\inner{x_i}{u_i^*}-\sum_{i=0}^p\inner{x_i}{v^*_{i-1}}\\
&=\sum_{i=0}^p\inner{y_i}{z_i^*}-\sum_{i=0}^p\inner{y_{i-1}}{w^*_{i}}=\sum_{i=0}^p\inner{y_i}{z_i^*}-\sum_{i=0}^p\inner{y_{i}}{w^*_{i+1}}\\
&=\sum_{i=0}^p\inner{y_i}{z_i^*-w^*_{i+1}}=\sum_{i=0}^p\inner{y_i}{x_i^*}=\inner{\bar y}{\bar x^*}_p,
\end{align*}
and $(\bar x,\bar x^*)\in (T_{p-})^*$.  The inclusion $(T^*)_{p-}\subset (T_{p+})^*$ is completely analogous.
\end{proof}

\begin{theorem}
Let $X$ be a reflexive Banach space and $T:X\tos X^*$ be a linear operator with closed graph and closed domain. Then the following are equivalent:
\begin{enumerate}
 \item $T$ is maximal $p$-monotone.
 \item $T^*$ is maximal $p$-monotone.
 \item $T^*$ is $p$-monotone.
\end{enumerate}

\end{theorem}
\begin{proof}
First observe that if $T$ is closed, then $T_{\pm}$ is also closed. 
If $T$ is maximal $p$-monotone then, by Proposition~\ref{pro:maxtp}, $T_{p-}$ is maximal monotone. Using Theorem~\ref{teo:bb} we conclude that $(T_{p-})^*$ is maximal monotone. Therefore $(T^*)_{p+}$ is monotone and, by Proposition~\ref{uno}, $T^*$ is $p$-monotone.  If $T^*$ is $p$-monotone then it is monotone and, by Theorem~\ref{teo:bb}, $T^*$ is maximal monotone. We now use Proposition~\ref{pro:pmonqmon}, item {\it 2}, to conclude that $T^*$ is maximal $p$-monotone. Finally, as we already proved {\it 1} $\to$ {\it 2}, if $T^*$ is maximal $p$-monotone then $(T^*)^*=T$ is maximal $p$-monotone.
\end{proof}



\begin{thebibliography}{10}

\bibitem{BBBRW07}
S.~Bartz, H.~H. Bauschke, J.~M. Borwein, S.~Reich, and X.~Wang.
\newblock {Fitzpatrick functions, cyclic monotonicity and {R}ockafellar's
  antiderivative}.
\newblock {\em Nonlinear Anal.}, 66(5):1198--1223, 2007.

\bibitem{BorBau2011}
H.~H. Bauschke, J.~M. Borwein, X.~Wang, and L.~Yao.
\newblock {The {B}rezis-{B}rowder theorem in a general {B}anach space}.
\newblock {\em J. Funct. Anal.}, 262(12):4948--4971, 2012.

\bibitem{Bauschke08}
H.~H. Bauschke and X.~Wang.
\newblock {An explicit example of a maximal 3-cyclically monotone operator with
  bizarre properties}.
\newblock {\em Nonlinear Analysis: Theory, Methods \& Applications},
  69(9):2875--2891, 2008.

\bibitem{BauWanYao2009}
H.~H. Bauschke, X.~Wang, and L.~Yao.
\newblock {Monotone linear relations: maximality and {F}itzpatrick functions}.
\newblock {\em J. Convex Anal.}, 16(3-4):673--686, 2009.

\bibitem{Birk79}
G.~Birkhoff.
\newblock {\em {Lattice theory}}, volume~25 of {\em {American Mathematical
  Society Colloquium Publications}}.
\newblock American Mathematical Society, Providence, R.I., third edition, 1979.

\bibitem{BueCot2017}
O.~Bueno and J.~Cotrina.
\newblock {On Maximality of Quasimonotone Operators}.
\newblock {\em Set-Valued and Variational Analysis}, May 2017.

\bibitem{BueCot16-2}
O.~Bueno and J.~Cotrina.
\newblock {The pseudomonotone polar for multivalued operators}.
\newblock {\em Optimization}, 66(5):691--703, 2017.

\bibitem{JC}
J.~Cotrina.
\newblock {A remark on the Br{\'e}zis--Browder theorem in finite dimensional
  spaces}.
\newblock {\em Optimization}, 0(0):1--6, 2017.

\bibitem{BM3}
M.~{Marques Alves} and B.~F. Svaiter.
\newblock {Maximal monotone operators with a unique extension to the bidual}.
\newblock {\em J. Convex Anal.}, 16(2):409--421, 2009.

\bibitem{BSML}
J.~E. Mart{\'i}nez-Legaz and B.~F. Svaiter.
\newblock {Monotone operators representable by l.s.c.\ convex functions}.
\newblock {\em Set-Valued Anal.}, 13(1):21--46, 2005.

\end{thebibliography}

\end{document}